\documentclass[10pt]{amsart}
\usepackage{graphicx}
\usepackage{amscd}
\usepackage{amsmath}
\usepackage{amsthm}
\usepackage{amsfonts}
\usepackage{amssymb}
\usepackage{mathrsfs}
\usepackage{enumerate}
\usepackage{amsrefs}
\usepackage[pagebackref,hypertexnames=false, colorlinks, citecolor=blue, linkcolor=red, pdfstartview=FitB]{hyperref}
\usepackage[latin1]{inputenc}

\newcommand\numberthis{\addtocounter{equation}{1}\tag{\theequation}}
\newcommand{\D}{\operatorname{\mathbb{D}}}

\newcommand{\Ds}{\operatorname{\mathscr{D}}}
\newcommand{\Rs}{\operatorname{\mathscr{R}}}

\newcommand{\N}{\operatorname{\mathbb{N}}}

\newcommand{\R}{\operatorname{\mathbb{R}}}
\newcommand{\Z}{\operatorname{\mathbb{Z}}}

\newcommand{\C}{\operatorname{\mathbb{C}}}
\newcommand{\Cal}{\operatorname{\mathcal{C}}}

\newcommand{\B}{\operatorname{\mathcal{B}}}
\newcommand{\A}{\operatorname{\mathcal{A}}}

\newcommand{\Fil}{\operatorname{\mathcal{F}}}
\newcommand{\hil}{\operatorname{\mathcal{H}}}

\newcommand{\e}{\operatorname{\varepsilon}}

\newcommand{\ol}{\overline }
\let\phi\varphi

\newtheorem{lemma}{Lemma}[section]
\newtheorem{theorem}[lemma]{Theorem}

\theoremstyle{definition}

\begin{document}
\author{Rapha\"el Clou\^atre}
\address{Department of Pure Mathematics, University of Waterloo, 200 University Avenue West,
Waterloo, Ontario, Canada N2L 3G1} \email{rclouatre@uwaterloo.ca}
\title[Completely bounded isomorphisms]{Completely bounded isomorphisms of operator algebras and similarity to complete isometries}
\begin{abstract}
A well-known theorem of Paulsen says that if $\A$ is a unital operator algebra and $\phi:\A\to B(\hil)$ is a unital completely bounded homomorphism, then $\phi$ is similar to a completely contractive map $\phi'$. Motivated by classification problems for Hilbert space contractions, we are interested in making the inverse $\phi'^{-1}$  completely contractive as well whenever the map $\phi$ has a completely bounded inverse. We show that there exist invertible operators $X$ and $Y$ such that the map
$$
XaX^{-1}\mapsto Y\phi(a)Y^{-1}
$$
is completely contractive and is ``almost" isometric on any given finite set of elements from $\A$ with non-zero spectrum.  Although the map cannot be taken to be completely isometric in general, we show that this can be achieved if $\A$ is completely boundedly isomorphic to either a $C^*$-algebra or a uniform algebra. In the case of quotient algebras of $H^\infty$, we translate these conditions in function theoretic terms and relate them to the classical Carleson condition.
\end{abstract}
\maketitle

\section{introduction}
A classical theorem of Sz.-Nagy (see \cite{nagy1947}) states that an invertible operator $T\in B(\hil)$ satisfying 
$$
\sup_{n\in \Z}\|T^n\|<\infty
$$
must be similar to a unitary. This was generalized by Dixmier (see \cite{dixmier1950}) to the setting of representations of amenable groups. A few years later, it was shown in \cite{nagy1959} again by Sz.-Nagy that any compact operator $T$ satisfying 
$$
\sup_{n\in \N}\|T^n\|<\infty
$$
must be similar to a contraction.  It was then natural to ask whether the conclusion still holds without the compactness assumption. This was shown to be incorrect by Foguel in \cite{foguel1964}.
In light of the von Neumann inequality for contractions, Halmos reformulated the problem as follows: is any polynomially bounded operator similar to a contraction? Recall that $T$ is said to be polynomially bounded if there exists some constant $C>0$ such that
$$
\|p(T)\|\leq C \|p\|_{\infty}
$$
for every polynomial $p$. This question turned out to be deeper and remained open for some time despite attracting significant interest (see \cite{aleksandrov1996}, \cite{bourgain1986}, \cite{carlson1994}, \cite{petrovic1992} among others) until Pisier proved that the answer is negative: there exists a polynomially bounded operator which is not similar to a contraction (see \cite{pisier1997}). In fact, what Pisier showed was that a certain operator is polynomially bounded but not \textit{completely} polynomially bounded. This distinction turned out to be the key insight  into this problem, and it was brought to light by the following result of Paulsen. The original statement can be found in \cite{paulsen1984}, but the improved version we state here is from \cite{paulsen1984PAMS}.

\begin{theorem}\label{t-paulsensim}
Let $\A$ be a unital operator algebra and $\phi:\A\to B(\hil)$ be a unital completely bounded homomorphism. Then, there exists an invertible operator $X$ with $\|X\|^2= \|X^{-1}\|^2=\|\phi\|_{cb}$  such that the map $a\mapsto X\phi(a)X^{-1}$ is completely contractive.
\end{theorem}

The problem we wish to address in this paper is related to this theorem. We are interested in the case where the homomorphism $\phi$ has a completely bounded inverse. We call such a map a \emph{completely bounded isomorphism}. This leads us to examine the behaviour and size of the norm of the inverse of the completely contractive map from Theorem \ref{t-paulsensim}.  Our motivation stems from the problem of classification of Hilbert space contractions through isomorphisms of some associated non-self-adjoint operator algebras such as the commutant (see \cite{clouatreSIM3}  and \cite{clouatre2013} for further details on this idea).

We make a few elementary comments before proceeding further. Assume that $\phi:\A\to \B$ is a unital completely bounded isomorphism. Then, the map 
$$
b\mapsto \|\phi^{-1}(b)\|
$$
gives rise to a new norm on $\B$ which is equivalent to the original one and which defines a new operator algebra structure by the Blecher-Ruan-Sinclair theorem (see \cite{blecher1990}). In particular, if we denote by $\B'$ this new operator algebra and by $\theta$ the inclusion of $\B$ into $\B'$, then $\theta$ and its inverse are completely bounded, and the map $\theta\circ \phi$ is completely isometric.  Our goal is to replace $\theta$ with spatial isomorphisms, that is conjugation by invertible operators from the ambient $B(\hil)$.

The plan of the paper is as follows. In Section 2,  we establish our most general positive result which says that if  $\phi:\A\to \B$ is a unital completely bounded algebra isomorphism, then there exist two invertible operators $X$ and $Y$ such that the map
$$
XaX^{-1}\mapsto Y\phi(a)Y^{-1}
$$
 is completely contractive and its inverse is ``almost" isometric on any given finite set of elements from $\A$ with non-zero spectrum  (we actually obtain a more general result, see Theorem \ref{t-semisimple} for the precise statement). We also show that we can obtain a complete isometry when the algebra $\A$ is completely boundedly isomorphic to either a $C^*$-algebra (Theorem \ref{t-c*}) or a uniform algebra (Theorem \ref{t-unif}). In Section 3 we exhibit a finite dimensional commutative algebra $\A$ and a unital completely bounded isomorphism on it for which this cannot be done, thus illustrating that Theorems  \ref{t-unif} and \ref{t-c*} do not hold in general. In Section 4 we make a connection with our motivating problem of classifying $C_0$ contractions by giving function theoretic interpretations of these Banach and operator algebraic conditions for quotients algebras of $H^\infty$ (see Theorem \ref{t-equivalencecarleson}). Finally, we end the paper in Section 5 with a discussion of possible further results.
 
 \emph{Acknowledgements:} The author wishes to thank Ken Davidson for his interest in the problem considered in this paper and for many stimulating conversations. He is also grateful to the referee for his careful reading of the paper and for a remark that improved the original statement of Theorem \ref{t-semisimple}.

\section{Positive results}
This section contains our main results. The first order of business is to prove the following lemma which is central to our approach for proving Theorem \ref{t-semisimple}.

\begin{lemma}\label{l-xnyn}
Let $\A\subset B(\hil_1)$ and $\B\subset B(\hil_2)$ be  unital operator algebras and $\phi:\A\to \B$ be a unital completely bounded isomorphism. Then, there exist invertible operators $X_n, Y_n$ such that
\begin{align*}
\|(a_{ij})\|
&\geq \|(Y_n \phi(a_{ij}) Y_n^{-1})\|\geq \|(X_n a_{ij} X_n^{-1})\|\geq \|(Y_{n+1} \phi(a_{ij}) Y_{n+1}^{-1})\|
\end{align*}
for every $(a_{ij})\in M_d(\A)$ and $d,n\in \N$. Moreover,  we have that 
$$
\|Y_{n+1}\|^2= \|Y_{n+1}^{-1}\|^2=\|X_n a X_n^{-1}\mapsto \phi(a)\|_{cb}
$$
and
$$
\|X_{n}\|^2= \|X_{n}^{-1}\|^2=\|Y_n \phi(a) Y_n^{-1}\mapsto a\|_{cb}
$$
for every $n\in \N$.
\end{lemma}
\begin{proof}
By Theorem \ref{t-paulsensim}, there exists an invertible operator $Y_1$ such that the map 
$$
\theta:\A\to Y_1 \B Y_1^{-1}
$$
$$
a\mapsto Y_1\phi(a)Y_1^{-1}
$$
is completely contractive and such that 
$$
\|Y_1\|^2=\|Y_1^{-1}\|^2=\|\phi\|_{cb}.
$$
We see that $\theta$ itself is a unital completely bounded isomorphism. We now apply Theorem \ref{t-paulsensim} to $\theta^{-1}$ to obtain an invertible operator $X_1$ such that the map 
$$
Y_1 \phi(a) Y_1^{-1}\mapsto X_1 a X_1^{-1}
$$
is completely contractive, with
$$
\|X_1\|^2=\|X_1^{-1}\|^2=\|Y_1\phi(a)Y_1^{-1}\mapsto a\|_{cb}.
$$
Repeating this procedure inductively, we obtain the operators $X_n$ and $Y_n$ for every $n\in \N$. Indeed, assume that the operators  $Y_{n-1}$ and $X_{n-1}$ have been constructed. A further application of Theorem \ref{t-paulsensim} to the unital completely bounded homomorphism given by
$$
X_{n-1}aX^{-1}_{n-1}\mapsto \phi(a)
$$ 
yields an invertible operator $Y_n$ such that the map 
$$
X_{n-1} a X_{n-1}^{-1}\mapsto Y_n \phi(a) Y_n^{-1}
$$
is completely contractive with
$$
\|Y_n\|^2=\|Y_n^{-1}\|^2=\|X_{n-1} a X_{n-1}^{-1}\mapsto \phi(a) \|_{cb}.
$$
Likewise, we find an invertible operator $X_n$ such that the map
$$
Y_{n} \phi(a) Y_{n}^{-1}\mapsto X_n a X_n^{-1}
$$
is completely contractive with
$$
\|X_n\|^2=\|X_n^{-1}\|^2=\|Y_{n} \phi(a) Y_{n}^{-1}\mapsto a \|_{cb}.
$$
Finally, notice that 
\begin{align*}
\|(a_{ij})\|& \geq \|(Y_1 \phi(a_{ij}) Y^{-1}_1)\|\geq \|(X_1 a_{ij} X_1^{-1})\|\\
&\geq \|(Y_2 \phi(a_{ij}) Y^{-1}_2)\|\geq \|(X_2 a_{ij} X_2^{-1})\|\\
&\geq \ldots \\
&\geq\|(Y_n \phi(a_{ij}) Y^{-1}_n)\|\geq \|(X_{n}a_{ij} X^{-1}_n)\|
\end{align*}
for every $(a_{ij})\in M_d(\A)$ and every $d,n\in \N$.
\end{proof}

We now arrive at our most general result. We actually obtain more than what was announced in the introduction.

Given $a\in \A$ we denote by $r(a)$ its spectral radius (that is, the radius of the smallest closed disk containing the spectrum of $a$). Recall than an element is said to be quasi-nilpotent if its spectrum consists only of the point $0$.

\begin{theorem}\label{t-semisimple}
Let $\A\subset B(\hil_1)$ and $\B\subset B(\hil_2)$ be unital operator algebras. Let $\phi:\A\to \B$ be a unital completely bounded isomorphism. Then, for any  $\e>0$ and any finite set $\A_0\subset \A$, there exist two invertible operators $X\in B(\hil_1)$ and $Y\in B(\hil_2)$ 
such that the map
$$
XaX^{-1}\mapsto Y\phi(a)Y^{-1}
$$
is a complete contraction and such that
$$
\|X a X^{-1}\|\leq (1+\e)\left(1+\e/\rho(\e) \right)\|Y \phi(a) Y^{-1}\|
$$
for every $a\in \A_0$, where $\rho(\e)$ is a positive constant depending only on $\e$.

Moreover, if the subset $\A_0$ contains no non-trivial quasi-nilpotent element, then we have the sharper inequality 
$$
\|X a X^{-1}\|\leq \left(1+\e/\rho \right)\|Y \phi(a) Y^{-1}\|
$$
for every $a\in \A_0$, where 
$$
\rho=\inf_{a\in \A_0}r(a).
$$
\end{theorem}
\begin{proof}
No generality is lost if we assume that $\A_0$ does not contain the zero element. Let $\{X_n\}_n, \{Y_n\}_n$ be the sequence of invertible operators obtained from Lemma \ref{l-xnyn}, so that for each $n\in \N$ we find
$$
\|(a_{ij})\|
\geq \|(Y_n \phi(a_{ij}) Y_n^{-1})\|\geq \|(X_n a_{ij} X_n^{-1})\|\geq \|(Y_{n+1} \phi(a_{ij}) Y_{n+1}^{-1})\|
$$
for every $(a_{ij})\in M_d(\A)$ and every $d\in \N$.  

Fix for a moment an element $a\in \A$. For each $n\in \N$ and $\lambda\in \C$, put 
$$
f_n(\lambda)=\|X_n (a+\lambda I)X_n^{-1}\|.
$$
Since the sequence of non-negative numbers $\{f_n(\lambda)\}_n$ is decreasing, it has a limit which we call $f(\lambda)$. For $\lambda_1,\lambda_2\in \C$, we have
\begin{align*}
f_n(\lambda_1+\lambda_2)&=\|X_n (a+\lambda_1 I)X_n^{-1}+\lambda_2\|\
\end{align*}
whence
$$
|f_n(\lambda_1+\lambda_2)-f_n(\lambda_1)|\leq |\lambda_2|.
$$
Hence, the same is true for the function $f$ and therefore it is continuous. Since the sequence of continuous functions $\{f_n\}_n$ is pointwise decreasing, we conclude that it converges uniformly to $f$ on the closed unit disc $\ol{\D}$.

Fix now $\e>0$. Applying the discussion of the previous paragraph to the elements of the finite subset $\A_0$, we see that there exists a natural number $N$ (depending only on $\e$) large enough so that
$$
\left|\|X_n (a+\lambda I) X_n^{-1}\|-\|X_m (a+\lambda I)X_m ^{-1}\| \right|< \e
$$
for every $a\in \A_0, \lambda\in \ol{\D}$ and $n,m\geq N$. 
Notice also that for every $a\in \A$, we can find a complex number $\zeta_a$ of absolute value $1$ such that
$$
r(a+\delta \zeta_a I)\geq \delta
$$
for every $\delta>0$.
Choose now $0<\delta<1$ small enough (depending only on $\e$) so that $\sigma<1+\e$, where 
$$
\sigma=\left(\sup_{a\in \A_0}\frac{\|Y_{N+1} (\phi(a)+\delta \zeta_a I)Y_{N+1}^{-1}\|}{\|Y_{N+1} \phi(a) Y_{N+1}^{-1}\|}\right)\left(\inf_{a\in \A_0}\frac{\|X_N (a+\delta \zeta_a I)X_N^{-1}\|}{\|X_N a X_N^{-1}\|} \right)^{-1}.
$$
Put
$$
\rho=\rho(\e)=\inf_{a\in \A_0}r(a+\delta \zeta_a I).
$$
We then have
$$
\|X_n (a+\delta \zeta_a I) X_n^{-1}\|\geq \rho
$$
for every $n\in \N$ and $a\in \A_0$.
These observations imply that
\begin{align*}
\frac{\|X_{N}(a+\delta \zeta_a I)X^{-1}_{N}\|}{\|X_{N+1}(a+\delta \zeta_a I)X_{N+1}^{-1}\|}&\leq1+ \frac{\e}{\|X_{N+1}(a+\delta \zeta_a I)X_{N+1}^{-1}\|}\\
&\leq 1+ \e/\rho
\end{align*}
for every $a\in \A_0$.
Therefore,
\begin{align*}
\|X_{N}(a+\delta \zeta_a I)X_N^{-1}\|&\leq \left(1+\e/\rho \right)\|X_{N+1}(a+\delta \zeta_a I)X_{N+1}^{-1}\|\\
&\leq \left(1+\e/\rho\right)\|Y_{N+1}(\phi(a)+\delta \zeta_a I)Y_{N+1}^{-1}\|
\end{align*}
whence
\begin{align*}
\|X_{N}a X_N^{-1}\|
&\leq \sigma\left(1+\e/\rho \right)\|Y_{N+1}\phi(a) Y^{-1}_{N+1}\|
\end{align*}
for every $a\in \A_0$.
We conclude that
\begin{align*}
\|X_{N}a X_N^{-1}\|
&\leq (1+\e)\left(1+\e/\rho \right)\|Y_{N+1}\phi(a) Y^{-1}_{N+1}\|
\end{align*}
for every $a\in \A_0$. Choosing $X=X_N$ and $Y=Y_{N+1}$ establishes the first statement. 

The second statement follows from the same idea, since under the assumption
$$
\inf_{a\in \A_0}r(a)>0
$$
the above proof works with $\delta=0$ (in which case $\sigma=1$).
\end{proof}

%
%

We wish to emphasize here that the second statement of the previous theorem says that
$$
XaX^{-1}\mapsto Y\phi(a)Y^{-1}
$$
is a complete contraction which is ``almost'' isometric on $\A_0$, as announced in the introduction. 
If $\A$ and $\B$ are commutative (not necessarily semi-simple) closed unital operator algebras, then one has an analogous result for the quotients of these algebras by their Jacobson radicals and for \emph{any} finite subset $\A_0$. This adaptation is straightforward and the details are left to the reader.  


We close this section with two results that apply to algebras $\A$ that are completely boundedly isomorphic to either a $C^*$-algebra or a uniform algebra (that is, a closed unital subalgebra of a commutative $C^*$-algebra). In those two cases we obtain a stronger conclusion than that of Theorem \ref{t-semisimple}. First, we record the following well-known facts.

\begin{lemma}\label{l-unif}
Let $\A\subset B(\hil)$ be a unital operator algebra and $\Fil$ be a uniform algebra.
\begin{enumerate}
\item[\rm{(i)}]  Every unital bounded homomorphism  $\phi:\A\to \Fil$ is completely contractive.

\item[\rm{(ii)}] Let $\theta:\Fil \to \A$ be a unital completely bounded isomorphism. Then, there exists an invertible operator $Y\in B(\hil)$ such that the map
$$
f\mapsto Y\theta(f)Y^{-1}
$$ 
is completely isometric.
\end{enumerate} 
\end{lemma}
\begin{proof}
The proof of (i) boils down to the fact that the norm of any element $f\in \Fil$ coincides with its spectral radius $r(f)$. Taking $n$-th roots from both sides of the inequality
$$
\|(\phi(a))^n\|=\|\phi(a^n)\|\leq \|\phi\| \|a^n\|.
$$
and letting $n$ tend to infinity yields
$$
\|\phi(a)\|=r(\phi(a))\leq r(a)\leq \|a\|
$$
for every $a\in \A$. Note that the first inequality in the line above holds because $\phi$ is a unital homomorphism. We conclude that $\phi$ is contractive. Now, $\Fil$ is a uniform algebra so $\phi$ takes values in a commutative $C^*$-algebra, and thus it  must be completely contractive (see \cite{loebl1975}), which shows (i).

For (ii), apply Theorem \ref{t-paulsensim} to obtain an invertible operator $Y$ such that the map
$$
f\mapsto Y\theta(f)Y^{-1}
$$ 
is  completely contractive.  Its inverse is also completely contractive by (i), and hence (ii) follows.
\end{proof}

In particular, we see by (i) above that if $\A$ and $\B$ are both completely isometrically isomorphic to uniform algebras, then any unital completely bounded isomorphism between them is necessarily completely isometric. Note that no conjugation by invertible elements is needed here. We now tackle the situation where $\A$ and $\B$ are merely completely boundedly isomorphic to a uniform algebra.

\begin{theorem}\label{t-unif}
Let $\A\subset B(\hil_1)$ and $\B\subset B(\hil_2)$ be unital operator algebras. Assume that there exists a unital completely bounded isomorphism $\theta:\Fil\to \A$ where $\Fil$ is a uniform algebra. Let $\phi:\A\to \B$ be a unital completely bounded isomorphism. Then, there exists two invertible operators $X\in B(\hil_1)$ and $Y\in B(\hil_2)$ with the property that the map
$$
XaX^{-1}\mapsto Y\phi(a)Y^{-1}
$$
is completely isometric.
\end{theorem}
\begin{proof}
By applying Lemma \ref{l-unif} (ii) to the maps $\theta$ and $\phi\circ \theta$, we find two invertible operators $X\in B(\hil_1)$ and $Y\in B(\hil_2)$ with the property that the maps
$$
f\mapsto Y\phi\circ \theta(f)Y^{-1}
$$
and
$$
f\mapsto X\theta(f)X^{-1}
$$
defined on $\Fil$ are completely isometric, which implies that the map
$$
XaX^{-1}\mapsto Y\phi(a)Y^{-1}
$$
is completely isometric as well.
\end{proof}

Obviously, this result applies to all algebras which are similar to uniform algebras, that is $\A=R\Fil R^{-1}$ for some invertible operator $R$.

We now turn to the case of $C^*$-algebras. Once again, we start with some well-known facts.

\begin{lemma}\label{l-c*}
Let $\Cal$ be a unital $C^*$-algebra and $\A\subset B(\hil)$ be a unital operator algebra. Let $\phi:\Cal \to \A$ be a unital completely bounded isomorphism. Then, there exists an invertible operator $Y$ such that the map
$$
c\mapsto Y\phi(c)Y^{-1}
$$ 
is a $*$-isomorphism onto $Y\A Y^{-1}$.
\end{lemma}
\begin{proof}
Theorem \ref{t-paulsensim} yields an invertible operator $Y$ such that the map
$$
c\mapsto Y\phi(c)Y^{-1}
$$ 
is completely contractive. Now, unital contractive linear maps of $C^*$-algebras are necessarily self-adjoint, so that this map is an injective $*$-homomor\-phism, and hence a $*$-isomorphism onto $Y\A Y^{-1}$.
\end{proof}
Note that the assumption on $\A$ in that lemma implies in particular that it must be closed in the norm topology of $B(\hil)$. Moreover, we see that the algebra $Y\A Y^{-1}$ is a $C^*$-algebra.

We now establish an analogue of Theorem \ref{t-unif}.

\begin{theorem}\label{t-c*}
Let $\A\subset B(\hil_1)$ and $\B\subset B(\hil_2)$ be unital operator algebras. Assume that there exists a unital completely bounded isomorphism $\theta:\Cal\to \A$ where $\Cal$ is a unital $C^*$-algebra. Let $\phi:\A\to \B$ be a unital completely bounded isomorphism. Then, there exists two invertible operators $X\in B(\hil_1)$ and $Y\in B(\hil_2)$ with the property that the map
$$
XaX^{-1}\mapsto Y\phi(a)Y^{-1}
$$
is completely isometric.
\end{theorem}
\begin{proof}
By applying Lemma \ref{l-c*} to the maps $\theta$ and $\phi\circ \theta$, we find two invertible operators $X\in B(\hil_1)$ and $Y\in B(\hil_2)$ with the property that the maps
$$
c\mapsto Y\phi\circ \theta(c)Y^{-1}
$$
and
$$
c\mapsto X\theta(c)X^{-1}
$$
defined on $\Cal$ are $*$-isomorphisms (and thus completely isometric), which in turn implies that the map
$$
XaX^{-1}\mapsto Y\phi(a)Y^{-1}
$$
is completely isometric as well.
\end{proof}

Once again, we mention that this theorem applies to any algebra which is similar to a unital $C^*$-algebra. This includes for instance commutative amenable unital operator algebras (as was recently shown in \cite{marcoux2013}).  It is unknown to us whether Theorem \ref{t-c*} holds for non-commutative amenable operator algebras. Note however that  the disc algebra is a uniform algebra which is not amenable (see \cite{pisier1997}), and thus the conclusion of Theorem \ref{t-c*} is not equivalent to amenability. 

The previous theorem holds in slightly greater generality. Indeed,  Theorem 2.6 of \cite{pitts2008} implies that the complete boundedness of the inverse is automatic once $\theta$ is assumed to be a bijective bounded homomorphism (we are grateful to David Pitts for bringing the results of \cite{pitts2008} to our attention). Furthermore, the condition of $\theta$ being completely bounded can be relaxed to mere boundedness if $\Cal$ satisfies the famous \emph{Kadison similarity conjecture} (see \cite{kadison1955}).  This long-standing conjecture is known to have an affirmative answer for irreducible or nuclear $C^*$-algebras (see \cite{christensen1981}) as well as for $C^*$-algebras that admit a cyclic vector (see \cite{haagerup1983}).

\section{A counter-example}

In this section we present an example that shows that Theorems \ref{t-unif} and \ref{t-c*} do not hold for general algebras. Indeed, we exhibit a unital completely bounded isomorphism $\Psi$ between finite dimensional commutative unital operator algebras for which it is impossible to find two invertible operators $X$ and $Y$ such that the map
$$
XaX^{-1}\mapsto Y\Psi(a)Y^{-1}
$$
is isometric, let alone completely isometric. 

We remark that these algebras almost satisfy the assumptions of Theorem \ref{t-unif}. In fact, under the extra assumption that they are semi-simple, their Gelfand transforms would yield the required map for Theorem \ref{t-unif} to apply. Therefore, our counter-example illustrates the importance of the absence of non-trivial quasi-nilpotent elements in that theorem.

 We will need the following two basic computational lemmas.

\begin{lemma}\label{l-rel}
Let $\alpha,\beta,\gamma,\alpha',\beta',\gamma',\delta'$ be complex numbers such that
\begin{equation}\label{e-norm2}
\left\|
\left(
\begin{array}{cc}
\alpha z_1 & \beta z_1+\gamma z_2 \\
0&0
\end{array}
\right)\right\|=
\left\|\left(
\begin{array}{cc}
\alpha'z_1& \beta'z_1+\gamma' z_2\\
0 &\delta' z_2
\end{array}
\right)\right\|
\end{equation}
for every $z_1,z_2\in \C$. 
Then, we have
\begin{equation}\label{e-gamma}
|\gamma|^2=|\gamma'|^2+|\delta'|^2
\end{equation}
and
\begin{equation}\label{e-alpha}
|\alpha|^2+|\beta|^2=|\alpha'|^2+|\beta'|^2.
\end{equation}
Moreover, we have
\begin{equation}\label{e-betagamma}
|\beta| |\gamma|\leq |\beta'| |\gamma'|.
\end{equation}
with equality only if $\alpha'\delta'=0$.
\end{lemma}
\begin{proof}
Computing the square of both sides of equation (\ref{e-norm2}) gives
\begin{align*}
|&\alpha z_1|^2+|\beta z_1+\gamma z_2|^2\\
&=\frac{1}{2}\left( |\alpha' z_1|^2+|\beta' z_1+\gamma' z_2|^2+|\delta' z_2|^2\right) \numberthis \label{e-norm3} \\
&+\frac{1}{2}\sqrt{(|\alpha' z_1|^2+|\beta' z_1+\gamma' z_2|^2+|\delta' z_2|^2)^2-4|\alpha'z_1|^2 |\delta'z_2|^2 } 
\end{align*}
for every $z_1,z_2\in \C$. In particular, we note that
\begin{equation}\label{e-normineq}
|\alpha z_1|^2+|\beta z_1+\gamma z_2|^2\leq  |\alpha' z_1|^2+|\beta' z_1+\gamma' z_2|^2+|\delta' z_2|^2
\end{equation}
for every $z_1,z_2\in \C$.
If $z_1=0$ and $z_2=1$, equation (\ref{e-norm3}) implies that
$$
|\gamma|^2=|\gamma'|^2+|\delta'|^2
$$
while setting $z_1=1$ and $z_2=0$ yields
$$
|\alpha|^2+|\beta|^2=|\alpha'|^2+|\beta'|^2
$$
so that equations (\ref{e-gamma}) and (\ref{e-alpha}) are established.
Next, pick $z_1,z_2$ with absolute value one and with the property that 
$$
|\beta z_1+\gamma z_2|=|\beta|+|\gamma|.
$$
A consequence of (\ref{e-normineq}) is that
\begin{align*}
|\alpha|^2+(|\beta|+|\gamma|)^2&\leq  |\alpha' |^2+|\beta' z_1+\gamma' z_2|^2+|\delta' |^2\\
&\leq |\alpha'|^2+(|\beta'|+|\gamma'|)^2+|\delta'|^2\\
&=|\alpha'|^2+|\beta'|^2+2|\beta'| |\gamma'|+|\gamma'|^2+|\delta'|^2\\
&=|\alpha|^2+|\beta|^2+|\gamma|^2+2|\beta'| |\gamma'|
\end{align*}
where the last line follows from equations (\ref{e-gamma}) and (\ref{e-alpha}). We infer that
$$
|\beta| |\gamma|\leq |\beta'| |\gamma'|.
$$
The proof is now completed by noting that inequality (\ref{e-normineq}) is strict if $\alpha'\delta'z_1 z_2\neq 0$.
\end{proof}

\begin{lemma}\label{l-diag}
Let $\alpha,\beta,\gamma,\alpha',\beta',\gamma',\delta'$ be complex numbers with $\gamma'\neq 0$. Assume that
\begin{equation}\label{e-norm4}
\left\|
\left(
\begin{array}{cc}
\alpha z_1 & \beta z_1+\gamma z_2 \\
0&0
\end{array}
\right)\right\|=
\left\|\left(
\begin{array}{cc}
\alpha'z_1& \beta'z_1+\gamma' z_2\\
0 &\delta' z_2
\end{array}
\right)\right\|
\end{equation}
for every $z_1,z_2\in \C$. Then, $\alpha'\delta'=0$.
\end{lemma}
\begin{proof}
In light of Lemma \ref{l-rel}, it suffices to show that $|\beta||\gamma|=|\beta'||\gamma'|$ in case where $\beta'\neq 0$.
Since $\gamma'\neq 0$ we can choose $z_1=1$ and $z_2=-\beta'/\gamma'$ and from (\ref{e-normineq}) we conclude that
\begin{align*}
|\alpha'|^2+\frac{|\delta'|^2|\beta'|^2}{|\gamma'|^2}\geq |\alpha |^2+\left|\beta-\frac{\gamma \beta'}{\gamma'}\right|^2.
\end{align*}
Therefore,
\begin{align*}
|\alpha'|^2+\frac{|\delta'|^2|\beta'|^2}{|\gamma'|^2}&\geq |\alpha |^2+|\beta|^2+\frac{|\gamma|^2 |\beta'|^2}{|\gamma'|^2}-2\frac{|\beta| |\gamma| |\beta'|}{|\gamma'|}
\end{align*}
which, after an application of equation (\ref{e-alpha}) and some easy simplifications, implies that
\begin{align*}
|\delta'|^2|\beta'|^2&\geq |\beta'|^2|\gamma'|^2+|\gamma|^2 |\beta'|^2-2|\beta| |\gamma| |\beta'| |\gamma'|.
\end{align*}
Since we assume $\beta'\neq 0$, this reduces to
$$
|\delta'|^2\geq |\gamma'|^2+|\gamma|^2 -2\frac{|\beta| |\gamma|  |\gamma'|}{|\beta'|}.
$$
In view of inequality (\ref{e-betagamma}), we find
$$
|\delta'|^2\geq |\gamma|^2 -|\gamma'|^2.
$$
By equation (\ref{e-gamma}), we see that equality must hold in (\ref{e-betagamma}) and the proof is complete.
\end{proof}

We are now ready to present our counter-example, the idea behind which is rather simple. Indeed, it exploits the fact that completely bounded linear isomorphisms of operator spaces can be much wilder than their multiplicative counterparts, by embedding an operator space as a corner of an operator algebra. We are indebted to Ken Davidson for suggesting this approach.

More precisely, consider the operator space $\mathscr{D}\subset M_2(\C)$ consisting of elements of the form
$$
\left(
\begin{array}{cc}
z_1 & 0 \\
0 & z_2
\end{array}
\right)
$$
along with the operator space $\mathscr{R}\subset M_2(\C)$ consisting of elements of the form
$$
\left(
\begin{array}{cc}
z_1 & z_2 \\
0 & 0
\end{array}
\right).
$$
Define now $\A_{\Ds}$ as the set of elements of the form
$$
\left(
\begin{array}{cc}
\lambda I_{\C^2} & s\\
0 & \lambda I_{\C^2}
\end{array}
\right)
$$
where $\lambda\in \C$ and $s\in \Ds$, and define $\A_{\Rs}$ analogously using elements of $\Rs$. Then, $\A_{\Rs}$ and $\A_{\Ds}$ are commutative unital subalgebras of $M_4(\C)$ (note that these algebras contain many non-trivial quasi-nilpotent elements). The map $\psi:\Rs\to \Ds$ defined as
$$
\psi\left(
\begin{array}{cc}
z_1 & z_2 \\
0 & 0
\end{array}
\right)=\left(
\begin{array}{cc}
z_1 & 0 \\
0 & z_2
\end{array}
\right)
$$
is easily seen to be a completely bounded linear isomorphism with completely bounded inverse, and therefore the induced map $\Psi:\A_{\Rs}\to\A_{\Ds}$ defined as
$$
\Psi\left(
\begin{array}{cc}
\lambda I_{\C^2} & s \\
0 & \lambda I_{\C^2}
\end{array}
\right)=\left(
\begin{array}{cc}
\lambda I_{\C^2} & \psi(s) \\
0 & \lambda I_{\C^2}
\end{array}
\right)
$$
is a unital completely bounded algebra isomorphism. We claim now that it is impossible to find two invertible elements $X,Y\in M_4(\C)$ satisfying
$$
\|XaX^{-1}\|=\|Y\Psi(a)Y^{-1}\|
$$
for every $a\in \A_{\Rs}$.

Assume to the contrary that there exist such matrices $X,Y\in M_4(\C)$. Upon multiplying on the left and on the right by unitary matrices (an operation which obviously preserves the above equality), we may assume that $X$ and $Y$ are upper triangular matrices. We write
$$
X=\left(
\begin{array}{cc}
X_{11}& X_{12} \\
0&X_{22}
\end{array}
\right),  Y=\left(
\begin{array}{cc}
Y_{11}& Y_{12}\\
0 &Y_{22}
\end{array}
\right)
$$
where $X_{11}, X_{22},Y_{11}, Y_{22}\in M_2(\C)$ are invertible upper triangular matrices. Consider now an element
$$
s=\left(
\begin{array}{cc}
z_1& z_2 \\
0 &0
\end{array}
\right)\in \Rs
$$
and put
$$
a=\left(
\begin{array}{cc}
0& s \\
0 &0
\end{array}
\right).
$$
Then,
$$
Y\Psi(a)Y^{-1}=\left(
\begin{array}{cc}
0 & Y_{11} \psi(s)Y_{22}^{-1} \\
0 &0
\end{array}
\right), XaX^{-1}=\left(
\begin{array}{cc}
0 & X_{11} s X_{22}^{-1} \\
0 &0
\end{array}
\right)
$$
and we must have
$$
\|X_{11} s X_{22}^{-1}\|=\|Y_{11}\psi(s)Y_{22}^{-1}\|
$$
for every $s\in \Rs$.
We therefore proceed to show that there exists no invertible upper triangular matrices $X_1,X_2,Y_1,Y_2\in M_2(\C)$ such that 
\begin{equation}\label{e-norm1}
\left\|
X_1\left(
\begin{array}{cc}
z_1 & z_2 \\
0&0
\end{array}
\right)X_2\right\|=
\left\| Y_1\left(
\begin{array}{cc}
z_1& 0\\
0 &z_2
\end{array}
\right)Y_2\right\|
\end{equation}
for every $z_1,z_2\in \C$. 

In order to reach a contradiction, we assume to the contrary that equation (\ref{e-norm1}) holds for every $z_1,z_2\in \C$.
Set
$$
X_1=\left(
\begin{array}{cc}
a_{11} & a_{12} \\
0& a_{22}
\end{array}
\right), X_2=\left(
\begin{array}{cc}
b_{11} & b_{12} \\
0& b_{22}
\end{array}
\right)
$$
and
$$
Y_1=\left(
\begin{array}{cc}
c_{11} & c_{12} \\
0& c_{22}
\end{array}
\right), Y_2=\left(
\begin{array}{cc}
d_{11} & d_{12} \\
0& d_{22}
\end{array}
\right)
$$
where $a_{ii},b_{ii},c_{ii},d_{ii}\neq 0$ for $i=1,2$. 
We first claim that $Y_1$ must be diagonal. Indeed, note that
$$
X_1\left(
\begin{array}{cc}
z_1 & z_2 \\
0&0
\end{array}
\right)X_2=\left(
\begin{array}{cc}
a_{11}b_{11}z_1 & a_{11}b_{12}z_1+a_{11}b_{22}z_2 \\
0&0
\end{array}
\right)
$$
and
$$
Y_1\left(
\begin{array}{cc}
z_1& 0\\
0 &z_2
\end{array}
\right)Y_2=\left(
\begin{array}{cc}
c_{11}d_{11}z_1& c_{11}d_{12}z_1+c_{12}d_{22}z_2\\
0 &c_{22}d_{22}z_2
\end{array}
\right).
$$
Therefore, Lemma \ref{l-diag} coupled with equation (\ref{e-norm1}) and the fact that $c_{11}d_{11}c_{22}d_{22}\neq 0$ shows that $c_{12}=0$.
Returning now to the main claim, by (\ref{e-betagamma}) we see that
$$
|a^2_{11}b_{12}b_{22}|\leq |c_{11}c_{12}d_{12}d_{22}|=0
$$
whence $b_{12}=0$. Moreover, we have
$$
|a_{11}b_{11}|^2=|a_{11}b_{11}|^2+|a_{11}b_{12}|^2=|c_{11}d_{11}|^2+|c_{11}d_{12}|^2
$$
and
$$
|a_{11}b_{22}|^2=|c_{12}d_{22}|^2+|c_{22}d_{22}|^2=|c_{22}d_{22}|^2
$$
by virtue of (\ref{e-gamma}) and (\ref{e-alpha}).
We compute 
$$
\left\|
X_1\left(
\begin{array}{cc}
z_1 & z_2 \\
0&0
\end{array}
\right)X_2\right\|^2=|a_{11}b_{11}|^2 |z_1|^2+|a_{11}b_{22}|^2 |z_2|^2
$$
and
\begin{align*}
&\left\|
Y_1\left(
\begin{array}{cc}
z_1& 0\\
0 &z_2
\end{array}
\right)Y_2
\right\|^2\\
&=\frac{1}{2}\left( (|c_{11}d_{11}|^2 +|c_{11}d_{12}|^2)| z_1|^2+|c_{22}d_{22}|^2|z_2|^2\right) \\
&+\frac{1}{2}\sqrt{( (|c_{11}d_{11}|^2+|c_{11}d_{12}|^2)| z_1|^2+|c_{22}d_{22}|^2|z_2|^2)^2-4|c_{11}d_{11}|^2 |z_1|^2 |c_{22}d_{22}|^2|z_2|^2 }.
\end{align*}
If $z_1=z_2=1$, then
\begin{align*}
\left\|
Y_1\left(
\begin{array}{cc}
1& 0\\
0 &1
\end{array}
\right)Y_2
\right\|^2&<  |c_{11}d_{11}|^2+|c_{11}d_{12}|^2+|c_{22}d_{22}|^2\\
&=|a_{11}b_{11}|^2+|a_{11}b_{22}|^2 \\
&=\left\|
X_1\left(
\begin{array}{cc}
1 & 1 \\
0&0
\end{array}
\right)X_2\right\|^2
\end{align*}
which is absurd. This establishes the claim, and shows that there are no invertible matrices $X,Y\in M_4(\C)$ with the property that the map 
$$
XaX^{-1}\mapsto Y\Psi(a)Y^{-1}
$$
is isometric on $\A_{\Rs}$.

\section{Quotient algebras of $H^\infty$}
In this section we give function theoretic interpretations of the various Banach and operator algebraic conditions appearing in the statements of Theorems \ref{t-semisimple}, \ref{t-unif} and \ref{t-c*} for quotients algebras of $H^\infty$ (see Theorem \ref{t-equivalencecarleson}). As mentioned in the introduction, the original motivation for the work done in this paper came from classification problems for $C_0$ contractions where these quotient algebras are of utmost importance.

We start by recalling some basic facts about the algebra $H^\infty$ of bounded
holomorphic functions on the open unit disc $\D$.
Let $H^2$ be the Hilbert space of functions
$$f(z)=\sum_{n=0}^\infty a_n z^n$$
holomorphic on the open unit disc, equipped with the norm
$$
\|f\|_{H^2}=\left(\sum_{n=0}^\infty |a_n|^2\right)^{1/2}.
$$
For any inner function $\theta\in H^\infty$, the space
$K_{\theta}=H^2\ominus \theta H^2$ is closed and invariant for $S^*$,
the adjoint of the shift operator $S$ on $H^2$. The operator
$S_{\theta}$ defined by $S_{\theta}^*=S^*|K_{\theta}$ is
called a model operator (or Jordan block). 
It is a consequence of the classical commutant lifting theorem of Sarason (see \cite{sarason1967}) that the algebra $H^\infty/\theta H^\infty$ is  isometrically isomorphic to $\{u(S_{\theta}):u\in H^\infty\}$ via the map
$$
u+\theta H^\infty \mapsto u(S_{\theta}).
$$
Using this identification, we view $H^\infty/\theta H^\infty$ as a unital operator algebra.
In particular, we have
$$
\|(u_{ij}+\theta H^\infty)\|_{M_d(H^\infty/\theta H^\infty)}=\|(u_{ij}(S_{\theta}))\|_{M_d(B(K_{\theta}))}
$$
for every $(u_{ij})\in M_d(H^\infty)$ and every $d\in \N$, while the spectrum of $u+\theta H^\infty$ coincides with that of $u(S_{\theta})$ for each $u\in H^\infty$. It is a well-known fact that the spectrum of $u(S_{\theta})$ consists of those complex numbers $\lambda$ with the property that
$$
(u-\lambda)H^\infty+\theta H^\infty\neq H^\infty
$$
which is equivalent to
$$
\inf_{z\in \D}\{|u(z)-\lambda|+|\theta(z)| \}=0
$$
by Carleson's corona theorem. A good reference for further detail on the subject is  \cite{bercovici1988}.
We now establish a characterization of quasi-nilpotent elements in $H^\infty/\theta H^\infty$. To simplify notation, we set
$$
b_{\lambda}(z)=\frac{z-\lambda}{1-\ol{\lambda}z}
$$
for every $\lambda$ in the open unit disc $\D$.

\begin{lemma}\label{l-spectheta}
Let $\theta\in H^\infty$ be an inner function and let $u\in H^\infty$ be an arbitrary function.
\begin{itemize}
\item[\rm{(i)}] The spectrum of $u+\theta H^\infty$ is $\{0\}$ if and only if given any sequence of points $\{z_n\}_n\subset \D$ satisfying $\lim_{n\to \infty}\theta(z_n)=0$ we also have $\lim_{n\to \infty}u(z_n)=0$.
\item[\rm{(ii)}] If $\theta=b_{\lambda}u$ for some $\lambda\in \D$, then the spectrum of $u+\theta H^\infty$ is contained in $\{u(\lambda),0\}$.

\end{itemize}
\end{lemma}
\begin{proof}
For the proof of (i), assume first that the spectrum consists only of $\{0\}$ and choose a sequence of points $\{z_n\}_n\subset \D$ with the property that $\theta(z_n)\to 0$. Then, every cluster point of the bounded sequence $\{u(z_n)\}_n$ belongs to the spectrum of $u+\theta H^\infty$, and thus $u(z_n)\to 0$. Conversely, assume $u$ has the announced property. If $\lambda$ is in the spectrum of $u+\theta H^\infty$ then there must exist a sequence $\{z_n\}_n\subset \D$ such that $\theta(z_n)\to 0$ and $u(z_n)\to \lambda$. The assumption on $u$ immediately implies that $\lambda=0$.

To show (ii), assume that $w\neq u(\lambda)$ and
$$
\inf_{z\in \D}\{|u(z)-w|+|\theta(z)|\}=0.
$$
Then, there exists a sequence $\{z_n\}_n\subset \D$ such that $\theta(z_n)\to 0$ and $u(z_n)\to w$. Since $w\neq u(\lambda)$, we see that $\{z_n\}_n$ does not converge to $\lambda$ and therefore the sequence given by $$u(z_n)=\theta(z_n)/b_{\lambda}(z_n)$$ must converge to $0$, whence $w=0$.
\end{proof}

We now proceed with the promised description of Banach and operator algebraic properties of $H^\infty/\theta H^\infty$ (see Theorems \ref{t-semisimple}, \ref{t-unif} and \ref{t-c*}) in terms of the properties of the function $\theta$. 

\begin{theorem}\label{t-equivalencecarleson}
Let $\theta\in H^\infty$ be an inner function.
\begin{enumerate}
\item[\rm{(i)}] The algebra $H^\infty/\theta H^\infty$ contains no non-trivial quasi-nilpotent elements if and only if $\theta$ is a Blaschke product with simple roots.

\item[\rm{(ii)}] The algebra $H^\infty/\theta H^\infty$ is a uniform algebra if and only if $\theta$ is an automorphism of the disc. In that case, the algebra is isomorphic to $\C$. In particular, $H^\infty/\theta H^\infty$ is a $C^*$-algebra if and only if it is a uniform algebra.

\item[\rm{(iii)}] The following statements are equivalent.
\begin{enumerate}
\item[\rm{(a)}] There exists a unital completely bounded isomorphism 
$$
\Phi: H^\infty/\theta H^\infty\to \Fil
$$
for some uniform algebra $\Fil$.

\item[\rm{(b)}] There exists a unital completely bounded isomorphism 
$$
\Phi: H^\infty/\theta H^\infty\to \Cal
$$
for some unital $C^*$-algebra $\Cal$.

\item[\rm{(c)}] the function $\theta$ is a Blaschke product whose roots $\{\lambda_n\}_n\subset \D$ satisfy the Carleson condition
$$
\inf_{n}\left\{\prod_{k\neq n}\left|\frac{\lambda_k-\lambda_n}{1-\ol{\lambda_k}\lambda_n} \right|\right\}>0.
$$
\end{enumerate}
\end{enumerate}
\end{theorem}
\begin{proof}
For the proof of (i), note that the function $\theta$ can be factored as
$$
\theta(z)=e^{it}\prod_{k}\left(\frac{\ol{-\lambda_k}}{|\lambda_k|}\frac{z-\lambda_k}{1-\ol{\lambda_k}z} \right)^{m_k}s_{\nu}(z)
$$
where $t\in \R$, $m_k\in \N$, $\lambda_k\in \D$ and $s_{\nu}$ is the inner function associated to some singular measure $\nu$ on the unit circle. Let 
$$
u(z)=\prod_{k}\left(\frac{\ol{-\lambda_k}}{|\lambda_k|}\frac{z-\lambda_k}{1-\ol{\lambda_k}z} \right)^{p_k}s_{\nu/2}(z)
$$
where $p_k=\max\{1,m_k-1 \}$ for each $k$.
Given a sequence of points $\{z_n\}_n\subset \D$ such that $\theta(z_n)\to 0$, it is easily verified that $u(z_n)\to 0$. This observation combined with part (i) of Lemma \ref{l-spectheta} yields that the spectrum of  $u+\theta H^\infty$ reduces to the singleton $\{0\}$. Note also that $u+\theta H^\infty \neq 0$ if $m_k>1$ for some $k$ or if $\nu$ is not identically $0$. This shows that if $H^\infty/\theta H^\infty$ contains no non-trivial quasi-nilpotent elements, then each $m_k$ must be equal to $1$ and $\nu$ must be identically zero. Conversely, assume that $\theta$ is a Blaschke product with distinct roots and that the spectrum of $u+\theta H^\infty$ is $\{0\}$. By part (i) of Lemma \ref{l-spectheta} again, we see that $u(z_n)\to 0$ whenever $\theta(z_n)\to 0$. By the assumption on $\theta$, this implies that $u\in \theta H^\infty$ whence $u+\theta H^\infty=0$.

Let us now turn to proving (ii). By virtue of (i), we see that unless $\theta$ is a Blaschke product with simple roots, then $H^\infty/\theta H^\infty$ contains non-trivial quasi-nilpotent elements and hence cannot be a uniform algebra. Assume therefore that $\theta$ is such a Blaschke product. If $\theta$ is an automorphism of the disc, then $H^\infty/\theta H^\infty$ is isomorphic to $\C$, and hence is trivially a uniform algebra. On the other hand, if $\theta$ is not an automorphism of the disc, then there exists a non-constant proper inner divisor $u$ of $\theta$ with the property that $\theta/u$ has only one root, say at $\lambda$. The operator $u(S_{\theta})$ is readily verified to have norm $1$ while
$$
r(u+\theta H^\infty)\leq |u(\lambda)|
$$
by Lemma \ref{l-spectheta} (ii).
By the maximum principle, we see that 
$$
r(u+\theta H^\infty)<1=\|u(S_{\theta})\|=\|u+\theta H^\infty\|
$$
and thus $H^\infty/\theta H^\infty$ is not a uniform algebra. The last statement of (ii) now easily follows from the fact that the algebra $H^\infty/ \theta H^\infty$ is commutative.

Let us finally prove (iii).
The fact that (b) implies (a) is trivial  since the algebra $\Cal$ must commutative.
We now show that (c) implies (b). Assume therefore that $\{\lambda_n\}_n$ satisfies the Carleson condition.  A consequence of the classical Carleson interpolation theorem (see \cite{carleson1958}) is that the map
$$
M_d(H^\infty)\to M_d(\ell^\infty)
$$
$$
(u_{ij})\mapsto ( \{u_{ij}(\lambda_n)\}_n)
$$
is surjective for every $d\in \N$. Note that the kernel of this map is precisely $M_d(\theta H^\infty)$. Now, a result due to Vasyunin implies that the operator $D=\bigoplus_n \lambda_n$ is similar to $S_{\theta}$  (see \cite{vasyunin1976} for the original paper or \cite{nikolskii1979} for an English translation, or see Theorem 4.4 of \cite{clouatre2011} for this precise statement). In particular, there exists some constant $\delta>0$ such that
$$
\delta \|(u_{ij}+\theta H^\infty)\|\leq \|(u_{ij}(D))\| \leq  \|(u_{ij}+\theta H^\infty)\|
$$
for every $(u_{ij})\in M_d(H^\infty)$ and every $d\in \N$. On the other hand, via the canonical shuffle we see that $(u_{ij}(D))$ is unitarily equivalent to $\bigoplus_n (u_{ij}(\lambda_n))$, whence
$$
\|(u_{ij}(D))\|=\sup_{n}\|  (u_{ij}(\lambda_n))\|_{M_d(\C)}.
$$
This shows that the map
$$
\Phi: H^\infty/\theta H^\infty \to \ell^\infty
$$
defined by
$$
\Phi(u+\theta H^\infty)=\{u(\lambda_n)\}_n
$$
is a unital completely bounded isomorphism.

To finish the proof of (iii), it remains to show that shows that (a) implies (c). Assume that there exists a unital completely bounded isomorphism
$$
\Phi: H^\infty/\theta H^\infty \to \Fil
$$
where $\Fil$ is a uniform algebra.
It follows in particular that
\begin{align*}
r(u+\theta H^\infty)&=r(\Phi(u+\theta H^\infty))=\|\Phi(u+\theta H^\infty)\|\\
&\geq \frac{1}{\|\Phi^{-1}\|} \|u+\theta H^\infty\|
\end{align*}
for every $u\in H^\infty$. For each $n\in\N$, let 
$$
u_n(z)=\prod_{k\neq n}\frac{\ol{-\lambda_k}}{|\lambda_k|}\frac{z-\lambda_k}{1-\ol{\lambda_k}z}.
$$
We must show that 
$$
\inf_n|u_{n}(\lambda_n)|>0.
$$ 
By part (ii) of Lemma \ref{l-spectheta}, we have that the spectrum of $u_n+\theta H^\infty $ is contained in $\{u_n(\lambda_n),0\}$, which in turn implies that
$$
r(u_n+\theta H^\infty)\leq |u_n(\lambda_n)|.
$$
On the other hand, the operator $u_n(S_{\theta})$ has norm $1$ as mentioned earlier, whence
\begin{align*}
|u_n(\lambda_n)|&\geq r(u_n+\theta H^\infty)\geq  \frac{1}{\|\Phi^{-1}\|} \|u_n+\theta H^\infty\|\\
&= \frac{1}{\|\Phi^{-1}\|}\|u_n(S_{\theta})\|= \frac{1}{\|\Phi^{-1}\|}
\end{align*}
and (ii) follows.
\end{proof}

Part (iii) of this theorem should be compared with the assumptions of Theorems \ref{t-unif} and \ref{t-c*}. Moreover, it is interesting to note that it offers yet another equivalent formulation of the classical Carleson condition, which to this day is still a ubiquitous notion in function theory.

\section{Possible extensions of our results}
We close this paper by mentioning various other sufficient conditions for the conclusions of Theorems \ref{t-semisimple}, \ref{t-unif} and  \ref{t-c*} to hold. This section contains no results proper, but we hope that the ideas described therein may eventually lead to enlarging the family of algebras to which those theorems apply. Moreover, the conditions we exhibit here share some connections with operator algebraic properties studied by other authors.

The first condition we wish to discuss is a particular kind of symmetry between a unital operator algebra $\A\subset B(\hil)$ and its ``adjoint'' $\A^*=\{a^*:a\in \A\}$. First recall that given a linear map $\theta$ acting between operator algebras, there is an associated linear map $\theta^*$ defined as $\theta^*(t^*)=\theta(t)^*$.  Now, we assume that $\A$ has the following property: whenever $\theta$ is a unital completely bounded homomorphism into $\A$, then  $\theta$ and $\theta^*$ are \emph{simultaneously} similar to completely contractive maps. In other words, there exists a single invertible operator $Z\in B(\hil)$ such that both the maps
$$
t\mapsto Z\theta(t)Z^{-1} \text{ and } t\mapsto Z\theta^*(t)Z^{-1}
$$
are completely contractive. 

The conclusion of Theorems \ref{t-unif} and \ref{t-c*} holds under this condition on $\A$, and we sketch the details of the proof. We refer the reader to the original statements found in Section 2 for the notation used here. Start by applying Theorem \ref{t-paulsensim} to find an invertible operator $Y$ such that
$$
\|(Y \phi(a_{ij}) Y^{-1})\|\leq \|(a_{ij})\|
$$
for every $(a_{ij})\in M_d(\A)$ and every $d\in \N$. Define $\theta$ to be the unital completely bounded homomorphism 
$$
Y \phi(a)Y^{-1}\mapsto a.
$$
By assumption on $\A$, we can find an invertible operator $Z$ such that
$$
\|(Z a_{ij}Z^{-1})\|\leq \|(Y \phi(a_{ij}) Y^{-1})\|
$$
and
$$
\|(Z a^*_{ji}Z^{-1})\|\leq \|(Y \phi(a_{ij}) Y^{-1})\|
$$
for every $(a_{ij})\in M_d(\A)$ and every $d\in \N$. Note now that
\begin{align*}
\|(a_{ij})\|^2& =r((a_{ij})^*(a_{ij}))\\
&\leq \|(Z\oplus\ldots \oplus Z) (a_{ij})^*(a_{ij}) (Z^{-1}\oplus\ldots \oplus Z^{-1})\|\\
&\leq \|(Za_{ji}^*Z^{-1})\| \|(Za_{ij}Z^{-1})\|\\
&\leq \|(Y \phi(a_{ij}) Y^{-1})\| ^2\\
&\leq \|(a_{ij})\|^2
\end{align*}
whence 
$$
\|(Y \phi(a_{ij}) Y^{-1})\|= \|(a_{ij})\|
$$
for every $(a_{ij})\in M_d(\A)$ and every $d\in \N$. We thus obtain a stronger result in this case since the map
$$
a\mapsto Y\phi(a)Y^{-1}
$$
is already completely isometric (in other words, we may choose $X$ to be the identity operator here). 

Unfortunately, this symmetry property seems rather mysterious and difficult to satisfy. For instance, Pisier has exhibited in \cite{pisier1998} an example of a pair of commuting operators $(T_1,T_2)$ on Hilbert space which are both similar to a contraction, but are not simultaneously similar to contractions. This illustrates the subtleties involved in such joint similarity problems.

We finish with another condition one can impose on the algebra $\A$ which is sufficient for a sharper version of the second statement of Theorem \ref{t-semisimple} to hold. Let us say
that an operator algebra $\A\subset B(\hil)$ has the \textit{conjugation with lower bound property} (CLBP) if there exits a constant $0<\delta\leq1$ such that 
$$
\|a\mapsto XaX^{-1}\|_{cb}\geq \delta \|X\| \|X^{-1}\|
$$
for every invertible $X\in B(\hil)$. This notion resembles that of ultraprimeness for general Banach algebras introduced in \cite{mathieu1989}, which has attracted some interest in past years (see \cite{cabrera1997}, \cite{runde2002} and \cite{strasek2005} for instance). The difference here is that we do not assume that the element $X$ lies in the algebra $\A$, and we restrict our attention to conjugation operators rather than general multiplication operators $M_{X,Y}(a)=XaY$. 

If $\A$ satisfies the CLBP, then the second statement in Theorem \ref{t-semisimple} holds for \emph{any} finite subset $\A_0\subset \A$, regardless of whether or not it contains non-trivial quasi-nilpotent elements. We briefly indicate how this can be done. We use the notation established in the proof found in Section 2. By construction of the sequence of invertible operators $\{X_n\}_n$, we have that
$$
\|a\mapsto X_n a X_n^{-1}\|_{cb}\leq 1
$$
for every $n\in \N$, 
and thus by the CLBP we have that
$$
\|X_n\| \|X_n^{-1}\|\leq \delta^{-1}
$$
for every $n\in \N$. Therefore, we have the estimate
$$
\|X_n a X_n^{-1}\|\geq \delta \|a\|
$$
for every $a\in \A$ and every $n\in \N$. This inequality makes it unnecessary to consider spectral radii in this approach, hence removing the need for the absence of quasi-nilpotent elements.

On the other hand, the CLBP is rather difficult to verify and concrete examples do not seem easy to come by so the range of applicability of the observations above seems limited at the moment. It is easy to see however that any algebra containing the rank-one operators has the CLBP and in fact 
$$
\|a\mapsto XaX^{-1}\|_{cb}=\|X\| \|X^{-1}\|
$$
in that case. Consequently, the sharper version of Theorem \ref{t-semisimple} certainly holds whenever the algebra $\A$ contains all finite rank operators.

\bibliography{/home/raphael/Dropbox/Research/biblio}
\bibliographystyle{plain}

\end{document}